\pdfoutput=1
\documentclass[11pt]{amsart}
\usepackage{epsfig,overpic}

\newtheorem{thm}{Theorem}[section]
\newtheorem{lem}[thm]{Lemma}

\theoremstyle{definition}

\newtheorem{note}[thm]{Note}

\theoremstyle{remark}

\newcommand{\R}{\mathbf{R}}
\newcommand{\Z}{\mathbf{Z}}
\newcommand{\cn}{\colon}
\newcommand{\ol}[1]{{\overline #1}}
\newcommand{\la}{\langle}
\newcommand{\ra}{\rangle}
 
\newcommand{\f}{\partial}
\newcommand{\w}{\wedge}
\newcommand{\AG}{{\mathbf{AG}}}
\newcommand{\C}{\mathcal{C}}

\renewcommand{\S}{\mathbf{S}}
\renewcommand{\(}{\left(}
\renewcommand{\)}{\right)}
\renewcommand{\tilde}{\widetilde}
\renewcommand{\div}{\operatorname{div}}

\DeclareMathOperator{\Area}{Area}
\DeclareMathOperator{\Length}{Length}
\DeclareMathOperator{\vol}{vol}

\begin{document}

\title[Total diameter and area]{Total diameter  and area of closed submanifolds}

\author[M.~Ghomi]{Mohammad Ghomi}
\address{School of Mathematics, Georgia Institute of Technology,
Atlanta, GA 30332}
\email{ghomi@math.gatech.edu}
\urladdr{www.math.gatech.edu/$\sim$ghomi}

\author[R.~Howard]{Ralph Howard}
\address{Department of Mathematics,
University of South Carolina,
Columbia, SC 29208}
\email{howard@math.sc.edu}
\urladdr{www.math.sc.edu/$\sim$howard}

\date{Last Typeset \today.}
\thanks{Research of the first named author was supported in part by NSF Grants DMS--1308777, DMS--0806305, and Simons Collaboration Grant 279374.}

\begin{abstract} 
The total diameter of a closed planar  curve $C\subset\R^2$
is the integral of its antipodal chord lengths. We show that this quantity is
bounded below by twice the area of $C$. Furthermore, when $C$ is convex or
centrally symmetric, the lower bound is twice as large. Both inequalities are
sharp and the equality holds in the convex case only when $C$ is a circle. We
also generalize these results to $m$ dimensional submanifolds of $\R^n$, where
the ``area"   will be defined in terms of the mod $2$ winding numbers of the
submanifold  about the $n-m-1$ dimensional affine subspaces  of $\R^n$.
\end{abstract}

\maketitle

\section{Introduction}
Integrals of chord lengths of closed curves in Euclidean space $\R^n$ are natural geometric quantities which have been studied since
Crofton (see Note \ref{note:BP} for historical background). More recently basic
inequalities involving these integrals  have been  used \cite{abrams:circle}
to settle conjectures of Freedman-He-Wang \cite{FHW:mobius} and O'Hara
\cite{ohara:energy1} on knot energies. See also \cite{EHL} for other
applications  to problems in physics. A fundamental result in this area
\cite[Cor. 3.2]{abrams:circle} \cite{gabor:mean} is  the  sharp inequality:
\begin{equation}\label{eq:0}
\int_0^L\left\|f(t+L/2)-f(t)\right\|\,dt\leq \frac{L^2}{\pi},
\end{equation}
where $f\colon\R/L\Z\to\R^n$ is a closed curve  parametrized by arc length.
Here, by contrast,  we develop sharp \emph{lower} bounds for the above integral, which we call the total diameter of $f$.
To describe these results,  let $M$ be a closed Riemannian $m$-manifold which has \emph{antipodal symmetry}, i.e., it admits a fixed point free isometry $\phi\colon M\to M$ such that $\phi^2$ is the identity map.
 Then, for every $p\in M$, we set 
$
p^*:=\phi(p),
$
and define the   \emph{total diameter}
of any mapping $f\colon M\to\R^n$ as
$$
TD(f):=\int_{M}\left\|f(p^*)-f(p)\right\|\,dp,
$$
which generalizes the integral in \eqref{eq:0}. Further note that the integrand here is the length of the line segment $f(p)f(p^*)$ which we call an \emph{antipodal chord} of $f$.
Next, to bound this quantity from below, we define the ``area" of $f$ as follows. Let $\mathbf{AG}=\mathbf{AG}_{n,n-m-1}$ denote the Grassmannian space of  affine $n-m-1$ dimensional subspaces of $\R^n$. There exists an invariant measure $dA$ on $\AG$ such that for any smooth compact $m+1$ dimensional embedded submanifold $S\subset\R^n$, the  $(m+1)$-volume of $S$ coincides with the integral   over all $\lambda\in\AG$ of the cardinality of $\lambda\cap S$, see \cite[p. 245]{santalo}.
Let $w_2(f,\lambda)$ denote the winding number mod $2$ of $f$ about $\lambda$. Then the \emph{area} of $f$ is defined as
 $$
A(f):=\int_{\lambda\in\AG} w_2(f, \lambda)\, dA,
$$
which is a variation on a similar notion studied by Pohl \cite{Pohl:Integral-formulas}, and
Banchoff and Pohl \cite{BP},  see Note \ref{note:BP}.
Note that $w_2(f,\lambda)$ is well-defined whenever $\lambda$ is disjoint from $f(M)$, and consequently $A(f)$ is well-defined when $f(M)$ has measure zero (e.g., $f$ is smooth).
Furthermore, when $f(M)$ is a closed embedded hypersurface, i.e., $m=n-1$ and $f$ is injective, $A(f)$ is simply the volume of the  region  enclosed by $f(M)$. 
\begin{thm}\label{thm:main}
Let $M$ be a closed $m$-dimensional Riemannian manifold with antipodal symmetry, and $f\colon M\to\R^n$ be a $\C^1$ isometric immersion. Then
\begin{equation}\label{eq:1}
 TD(f) \geq 2A(f),
 \end{equation}
 and equality holds if and only if $f(p)=f(p^*)$ for all $p\in M$ (i.e.,  both sides of the inequality vanish).   Furthermore, if $f$ is centrally symmetric or convex, then
\begin{equation}\label{eq:2}
TD(f) \geq 2(m+1) A(f),
\end{equation}
and equality holds if and  only if $f$ is a sphere. 
\end{thm}

Here \emph{centrally symmetric}   means that, after a translation, $f(p) = -f(p^*)$ for all $p\in M$. By \emph{convex} we mean that $f$ traces injectively the boundary of a convex set in an $m+1$ dimensional affine subspace of $\R^n$. Further, when  this set is  a ball, we say that $f$ is a \emph{sphere}.
 Note that, when $m=1$, we may identify $M$ with the circle $\R/L\Z$, in which case \eqref{eq:2} yields
 $$
\int_0^L\left\|f(t+L/2)-f(t)\right\|\,dt\geq 4 A(f) 
 $$
 for convex planar curves $f\colon \R/L\Z\to\R^2$ parametrized by arc length.
 This together with \eqref{eq:0} in turn yields $A(f)\leq L^2/(4\pi)$, which is the classical isoperimetric inequality. For another quick application of Theorem \ref{thm:main}, note that if $D$ denotes the diameter of a planar curve $f$, or the maximum length of all its chords, then  
 $LD\geq TD(f)$ and thus
 \eqref{eq:2} implies that 
$
LD\geq 4A(f)
$
when $f$ is convex or centrally symmetric. In the convex case, this is a classical inequality due to  Hayashi \cite{hayashi,scott}. 
Another interesting feature of the above theorem is that  equality in \eqref{eq:1} is never achieved when $f$ is simple or injective; however, we will show in Section \ref{sec:example} that the strict inequality is still sharp for simple curves. More specifically,  we construct a family of simple closed planar curves $f_n$ such that $TD(f_n)/A(f_n)\to 2$ as $n\to\infty$. This also shows that  \eqref{eq:2} does not hold without the convexity or the symmetry conditions.

The proof of Theorem \ref{thm:main} unfolds as follows. First, in Section
\ref{sec:toplem}, we use some basic degree theory to show that each affine
space $\lambda\in \AG$ with $w_2(f,\lambda)\neq 0$ intersects an antipodal
chord of $f$.  Then in Section \ref{sec:ProofOf2} we integrate the Jacobian of
a natural map parametrizing the antipodal chords of $f$ to obtain \eqref{eq:1}
fairly quickly; see also Note \ref{note:geometric} for an intuitive geometric
proof of \eqref{eq:1} for planar curves. The same Jacobian technique also
yields the proof of the symmetric case of \eqref{eq:2} in Section
\ref{sec:symmetric} with a bit more work. Next we consider the convex case of
\eqref{eq:2} in Section \ref{sec:convex}. Here, when $m\geq 2$, the rigidity
of convex hypersurfaces (see Lemma \ref{lem:rigid}) reduces the problem to the
symmetric case already solved in the Section \ref{sec:symmetric}. The case of
$m=1$ or convex planar curves, on the other hand, requires more work, and
surprisingly enough constitutes the hardest part of Theorem \ref{thm:main}.

\begin{note}[Regularity of curves in Theorem \ref{thm:main}] 
In the case where
$m=1$, or $f\colon M\simeq\R/L\Z\to\R^n$ is a closed curve, the regularity
requirement for $f$ in Theorem \ref{thm:main} may be relaxed. In fact it
suffices to assume in this case that $f$ is rectifiable and  parametrized by
arc length, i.e., for every interval $I\subset\R$, the length of $f(I)$
coincides with that of $I$. Then $f$ will be Lipschitz and thus absolutely
continuous. So,  by a theorem of Lebesgue,   it is differentiable almost
everywhere and satisfies  the fundamental theorem of calculus. Consequently,
all  arguments below apply to $f$ once it is understood that the expressions
involving $f'(t)$ are meant to hold for almost all $t$. 
\end{note}

\begin{note}[Historical  background]\label{note:BP}
The first person to study integrals of chord lengths of planar curves
seems to have been Crofton in the remarkable papers
\cite{Crofton:Prob,Crofton:local} where he  considers 
the length of chords cut off by a random line, and also the powers of these lengths.  Furthermore these papers give the invariant measure on
the space of lines in $\R^2$, that is the measure $dA$ on $\mathbf{AG}_{2,1}$ mentioned above.
For more on these results and their history 
see Santal\'o's book  \cite[Chap.~4]{santalo}.
Using winding numbers to generalize the notion of  volume enclosed by a simple 
curve or surface seems to have
originated in the works of Rad\'o \cite{Rado:Lemma,Rado:Surface-Area,Rado:Book}, where he considers 
nonsimple curves in $\R^2$ and the maps from the two dimensional sphere into
$\R^3$.  
The idea of relating integrals of linking numbers with affine subspaces to
integrals of chord lengths for curves in Euclidean spaces is due to Pohl
\cite{Pohl:Integral-formulas}.  The  generalization of these linking
integrals to higher dimensional submanifolds, and pointing out that they
extend the notion of enclosed volume to higher codimensions, appear  in the paper of Banchoff and Pohl \cite{BP}.
In contrast to our definition of $A(f)$ above, Banchoff and Pohl define
the area as $\int_{\lambda\in\AG} w^2(f,
\lambda) dA$, where $w(f,\lambda)$ is the winding number of $f$ about
$\lambda$ (which is well-defined only when $M$ is orientable). Using this
concept, they generalize the classical isoperimetric inequality to
nonsimple curves and higher dimensional (and codimensional) submanifolds.
\end{note}

\section{A  Topological Lemma}\label{sec:toplem}
The proofs of both inequalities in Theorem \ref{thm:main} hinge on the
following purely topological fact. Let us first review  the general definition
of winding number mod~$2$. Here $M$ denotes  a closed topological
$m$-manifold, $f\colon M\to\R^n$ is a continuous map, and $\lambda\subset
\R^n-f(M)$ is an  $n-m-1$ dimensional affine subspace. Let
$\lambda'\subset\R^n$ be an $m+1$ dimensional affine subspace which is
orthogonal to $\lambda$, and $\pi\colon\R^n\to\lambda'$ be the orthogonal
projection. Then $\pi(\lambda)$ consists of a single point, say $o$, which is
disjoint from $\pi(f(M))$, and we set
$$
w_2(f,\lambda):=w_2(\pi\circ f, o).
$$
It remains then to define $w_2(\pi\circ f, o)$. To this 
end we may identify $\lambda'$ with $\R^{m+1}$ and assume 
that $o$ is the origin. Then $r:= \pi\circ f/\|\pi\circ f\|$ yields a mapping $M\to\S^m$, and we set 
$$
w_2(\pi\circ f, o):=\deg_2(r),
$$
the degree mod $2$ of $r$. If $r$  is smooth (which we may assume it is after a perturbation), $\deg_2(r)$ is simply the number of points mod $2$ in $r^{-1}(q)$ where $q$ is any regular value of $r$. Alternatively, $\deg_2(r)$ may be defined in terms of the $\Z_2$-homology of $M$. In particular  $w_2(f,\lambda)$ is well-defined even when $M$ is not orientable. See \cite[p. 124]{or:degree} for more background on mod $2$ degree theory.

We say that a topological manifold $M$ has \emph{antipodal symmetry} provided that there exists a fixed point free homeomorphism $\phi\colon M\to M$ with $\phi^2=id_M$. Then for every $p\in M$, the corresponding \emph{antipodal point} is $p^*:=\phi(p)$. An \emph{antipodal chord} of   $f\colon M\to\R^n$ is a line segment connecting $f(p)$ and $f(p^*)$ for some $p\in M$. 

\begin{lem}\label{lem:degree}
Let $M$ be a closed $m$-manifold with antipodal symmetry,  $f\colon M\to\R^n$ be a continuous map, and $\lambda\subset \R^n-f(M)$ be an  $n-m-1$ dimensional affine space such that 
 $w_2(f,\lambda)\neq 0$. Then  an antipodal chord of $f$   intersects $\lambda$.
\end{lem}

In particular note that, according to this lemma, any point in the region enclosed by a simple closed curve $C\subset\R^2$ intersects some antipodal chord of $C$.

\begin{proof}
Let $\pi\colon\R^n\to\lambda'$, and $o:=\pi(\lambda')$ be as discussed above. 
Suppose towards a contradiction that no antipodal chord of $f$ passes through $\lambda$. Then no antipodal chord of $\pi\circ f$ passes through $o$, and $w_2(\pi\circ f, o)=w_2(f,\lambda)\neq 0$.  Now set $f_0:=\pi\circ f$, identify $\lambda'$ with $\R^{m+1}$ and $o$ with the origin. 
Further,  define $f_1\colon M\to\R^{m+1}$ by 
$$
f_1(p):=\frac12\big(f_0(p)+f_0(p^*)\big).
$$
Note that $F\colon M\times I\to\R^{m+1}$ given by
$$
F(p,t):=\(1-\frac t2\) f_0(p)+\frac t2\,  f_0(p^*)
$$
gives a homotopy between $f_0$ and $f_1$ in the complement of $o$. Thus 
$$
w_2(f_1,o)=w_2(f_0,o)=w_2(f,\lambda)\neq 0.
$$ 
On the other hand, $f_1(p)=f_1(p^*)$. Thus if we set $\tilde M:=M/\phi$ and let $\pi\colon M\to\tilde M$ be the corresponding covering map, then $f_1$ induces a mapping $\tilde f_1\colon \tilde M\to\R^{m+1}$ such that $\tilde f_1\circ\pi=f_1$. Now let $r_1:=f_1/\|f_1\|$, $\tilde r_1:=\tilde f_1/\|\tilde f_1\|$. Then 
$$
\tilde r_1\circ\pi=r_1.
$$
 Further note that since $M$ is a double covering of $\tilde M$, $\deg_2(\pi)=0$. Thus the multiplication formula for mod 2 degree yields that
$$
w_2(f_1,o)=\deg_2(r_1)=\deg_2(\tilde r_1) \deg_2(\pi)=\deg_2(\tilde r_1) \cdot 0=0,
$$
and we have the desired  contradiction.
\end{proof}

\section{Proof of  \eqref{eq:1}}\label{sec:ProofOf2}
Equipped with the topological lemma established above, we now proceed towards proving the first inequality in Theorem \ref{thm:main}. To this end, for any $p\in M$, let
$$
\ol f(p):=f(p^*)-f(p),
$$
be the \emph{antipodal vector} of $f$ at $p$, and define $F\cn M\times [0,1]\to \R^n$
by
$$
F(p,t):=(1-t)f(p) + tf(p^*)=f(p)+t\ol f(p).
$$
Note that $F$ covers each antipodal chord of $f$ twice, and thus by Lemma \ref{lem:degree}, intersects each  $\lambda\in\AG$ with   $w_2(\lambda,f)\neq 0$ at least twice. Consequently 
\begin{equation}\label{eq:2af}
2A(f)\leq \int_0^1\int_M J(F) dp \,dt,
\end{equation}
where $J(F)=J(F)(p,t)$ denotes the Jacobian of $F$.
To compute $J(F)$, let  $e_j$, $1\leq j\leq m$, be an orthonormal basis
of $T_pM$, and $\partial/\partial t$ be the standard basis for $[0,1]$. Then $\{e_j,\partial/\partial t\}$ forms an orthonormal basis for $T_{(p,t)}(M\times [0,1])$, and thus $J(F)$ is the volume of the parallelepiped, or the norm of the $(m+1)$-vector, spanned by the derivatives of $F$ with respect to $\{e_j,\partial/\partial t\}$; see \cite{morgan:book} or \cite{krantz&parks:book2} for more background on $m$-vectors and exterior algebra. 
More specifically, if we set
$$
F_j:=dF(e_j)\quad\text{and}\quad F_t:=dF\(\frac{\partial}{\partial t}\)=\frac{\partial F}{\partial t}=\ol f,
$$
then we have
\begin{equation}\label{eq:main}
J(F) =\| F_1\w\cdots\w F_m\w \ol f\|\le \| F_1\|  \cdots  \|F_m\| \|\ol f\|.
\end{equation}
Next note that since $f$ is an isometric immersion, and $\phi$ is an isometry,
$$
\epsilon_j:=df_p(e_j) \qquad \text{and} \qquad \epsilon^*_j:=d(f\circ\phi)_p(e_j)
$$
are each orthonormal as well, and we have
\begin{equation}\label{eq:dfej}
F_j=(1-t)df_p(e_j)+t\,d(f\circ\phi)_p(e_j) =(1-t)\epsilon_j+t\epsilon_j^*.
\end{equation}
Thus 
\begin{equation}\label{eq:dFejsquare}
\|F_j\|^2=1+2t(1-t)(\la \epsilon_j,\epsilon_j^*\ra-1)\leq 1.
\end{equation}
So it follows that
\begin{equation}\label{eq:JFptfbar}
J(F)(p,t)\le \|\ol f(p)\|,
\end{equation}
which in turn yields
\begin{equation*}
2A(f)\leq \int_0^1\int_M J(F)\,dp\,dt\le \int_M\|\ol f(p)\|\,dp=TD(f)
\end{equation*}
as desired. 
Next, to establish the sharpness of \eqref{eq:1},
suppose that  the first and last terms of the above expression are equal. Then the middle two terms will be equal as well. This in turn implies that equality holds in \eqref{eq:JFptfbar}. Now \eqref{eq:main} yields that equality must hold in \eqref{eq:dFejsquare}, which can happen only if $\langle\epsilon_j, \epsilon_j^*\rangle =1$. So $\epsilon_j=\epsilon_j^*$, and we have
$$
0=\epsilon_j^*-\epsilon_j=d(f\circ\phi)_p(e_i)-df_p(e_j)=d\ol f_p(e_j).
$$
Hence $\ol f\equiv a$ for some constant vector $a\in\R^{n}$. But $\ol f(p)=-\ol f(p^*)$. Thus $a=0$, which yields that $f(p^*)=f(p)$ as claimed.

\begin{note}[A geometric proof of \eqref{eq:1} for planar curves]\label{note:geometric}
Here we describe an alternate proof of \eqref{eq:1} for planar curves $f\colon \R/L\Z\to\R^2$, which is more elementary and transparent, but may not yield the sharpness of the inequality so easily.  Divide the circle $\R/L\Z$ into $2n$ arcs
$C_i$ of length $\Delta t:=L/2n$,
$i\in\Z/(2n\Z)$, and note
that $C_i^*=C_{i+n}$, i.e., $C_{i+n}$ is the antipodal reflection of $C_i$ given by the correspondence $t\mapsto t^*:=t+L/2$. Let $R_i$ be the  region covered by all the antipodal chords connecting $f(C_i)$ and $f({C_i}^*)$, and set $A_i:=\Area(R_i)$.
\begin{figure}[h]
\begin{center}
\begin{overpic}[height=1.3in]{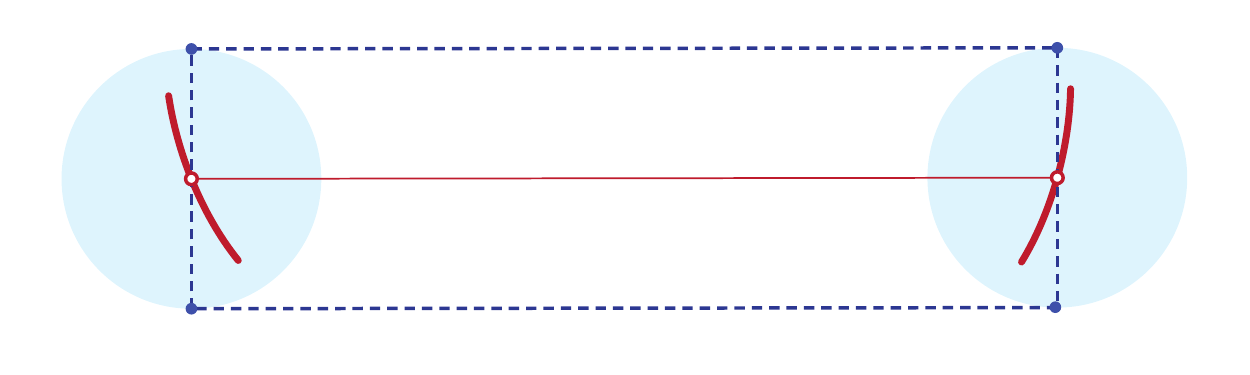} 
\put(6.5,14){\small$f(m_i)$}
\put(85,14){\small$f(m_i^*)$}
\end{overpic}
\caption{}
\label{fig:glasses2}
\end{center}
\end{figure}
 By Lemma~\ref{lem:degree}, $R_i$ covers all points $x\in\R^2$ with $w_2(f,x)\neq 0$, as $i$ ranges from $1$ to $n$. Thus 
\begin{equation}\label{eq:Ai}
A(f)\leq\sum_{i=1}^n A_i=\frac{1}{2}\sum_{i=1}^{2n} A_i.
\end{equation}
Next we are going to derive an upper bound for each $A_i$. Let $m_i$ denote
the midpoint of $C_i$, and consider the disk $D_i$ of radius $\Delta t/2$
centered at $f(m_i)$.  Note that, since $f$ is parametrized by arc length, $f(C_i)\subset D_i$. Now consider the diameters of $D_i$ and $D_i^*:=D_{i+n}$ which are orthogonal
to the antipodal chord $f(m_i)f(m_i^*)$, and let $R_i'$ be the rectangle formed by connecting the end
points of these diameters, see Figure \ref{fig:glasses2}. Then 
$$
R_i\subset
R_i'\cup D_i\cup D_{i}^*,
$$
because  $R_i'\cup D_i\cup D_{i}^*$ is a convex set; indeed it is the convex hull of $D_i\cup D_{i}^*$. Thus setting $A_i':=\Area(R_i')$, we have
$$
A_i\leq A_i'+\pi \left(\frac{\Delta t}{2}\right)^2=\|f(m_i)-f(m_i^*)\|\Delta t +\frac\pi4 (\Delta t)^2.
$$
So it follows that
$$
 A(f)\leq \frac{1}{2}\sum_i^{2n}\(\|f(m_i)-f(m_i^*)\|\Delta t +\frac\pi4 (\Delta t)^2\).
$$
  Taking the limit of the last expression as $\Delta t\to 0$   yields \eqref{eq:1}.
\end{note}

\section{Proof of \eqref{eq:2} in the Centrally Symmetric Case}\label{sec:symmetric}
Here we continue using the same notation established in the last section.
If $f$ in Theorem \ref{thm:main}  is symmetric, then after a translation we may assume that 
$
f(p^*)=-f(p)
$
for all $p\in M$, or $f\circ\phi =-f$ on $M$, which yields that 
$
\epsilon_j^*=-\epsilon_j.
$
Consequently, it follows from \eqref{eq:dfej} that 
\begin{equation}\label{eq:dFej}
F_j=(1-2t)\epsilon_j,
\end{equation}
 and thus by \eqref{eq:main} we have
\begin{equation}\label{eq:JFpt}
J(F)\le \| F_1\|  \cdots  \|F_m\|\left\|\ol f\right\|=|1-2t|^m\|\ol f\|.
\end{equation}
So
\begin{equation}\label{eq:intJF}
\int_0^1\int_M J(F)\,dp\,dt \le\int_0^1|1-2t|^m\,dt \int_M\|\ol f(p)\|\,dp 
=\frac{1}{m+1}\int_M\|\ol f(p)\|\,dp.
\end{equation}
This together with \eqref{eq:2af} yields
$$
2(m+1) A(f)\leq (m+1)\int_0^1\int_M J(F)\,dp\,dt \le\int_M\|\ol f(p)\|\,dp=TD(f)
$$
as claimed. 
To establish the sharpness of \eqref{eq:2},
suppose that the first and last terms of the above expression are equal. Then the middle two terms will be equal as well. This in turn yields that the first and last terms of  \eqref{eq:intJF} are equal, and so equality holds between the first two terms of \eqref{eq:intJF}. It then follows that  equality holds in \eqref{eq:JFpt}.  This can happen only if 
$$
\la\ol f,F_j\ra=0.
$$ 
But, since we have assumed that $f$ is symmetric with respect to the origin, $\ol f=2f$. Furthermore, by \eqref{eq:dFej},  $F_j$ is parallel to $\epsilon_j=df_p(e_j)$. 
Thus $f(p)$ is orthogonal to $df_p(e_j)$, which yields that
$$
\Big(e_j\(\|f\|^2\)\Big)(p)=2\langle f(p), df_p(e_j)\rangle=0.
$$
So $\|f\|$ is constant on $M$, which means that $f$ is a sphere. 

\begin{note}[A quick proof  of \eqref{eq:2} for centrally symmetric embedded hypersurfaces] 
Let $D$ be a bounded domain in $\R^n$ with $\C^1$
boundary $\f D$, and $X\colon\R^n\to\R^n$ be the position vector field given by $X(p):=p$.  Then
$\div X=n$. If $\nu$ is the outward normal along $\partial D$, then by the divergence theorem
$$
\vol(D)=\frac1n \int_D \div X\,dV = \frac1n\int_{\f D} \la X,\nu\ra\,dA
\le \frac1n \int_{\f D} \|X\|\,dA
$$
where $dV=dx^1\w \cdots\w dx^n$, and $dA$ is the surface area measure
on $\partial D$.
Thus if $M$ is a Riemannian manifold and $f\cn M\to \f D$ is an isometry,
then
$$
\vol(D)\le \frac1n \int_{M} \|f(p)\|\,dA.
$$
If $D$ is symmetric about the origin, $f(p)-f(p^*)=2f(p)$, and thus the above  inequality reduces to \eqref{eq:2} in the case where $m=n-1$, and $f$ is symmetric and injective.
\end{note}

\section{Proof of \eqref{eq:2} in the Convex Case}\label{sec:convex}
Here we need to treat the case of convex curves ($m=1$) separately from that of higher dimensional convex hypersurfaces $(m\geq 2)$, because convex hypersurfaces are rigid when $m\geq 2$, and consequently the argument here reduces to the symmetric case considered earlier; however, for convex curves (which are more flexible) we need to work harder.

\subsection{Convex hypersurfaces ($m\geq 2$)}
Recall that when we say $f\colon M\to\R^n$ is convex, we mean that $f$ maps $M$ injectively into the boundary of a convex subset of an $m+1$ dimensional affine subspace of $\R^n$, which  we may identify with $\R^{m+1}$.
Since we have already treated the symmetric case, it is enough to show that:

\begin{lem}\label{lem:rigid}
Let $M$ be a closed Riemannian $(m\geq 2)$-manifold with antipodal symmetry, and $f\colon M\to\R^{m+1}$ be a $\C^1$ isometric convex embedding.  Then $f$ is centrally symmetric.
\end{lem}

To prove the above lemma we need to recall the basic rigidity results for convex hypersurfaces. A (closed) \emph{convex hypersurface} is the boundary of a compact convex subset of $\R^{m+1}$ which has nonempty interior. We say that a class of convex hypersurfaces of $\R^{m+1}$ is \emph{rigid} if any isometry between a pair of members in that class can be extended to an isometry of $\R^{m+1}$. An equivalent formulation is that for every pair of convex isometric embeddings $f, g\colon M\to\R^{m+1}$, there exists an isometry $\rho$ of $\R^{m+1}$ such that $\rho\circ f=g$. That all convex surfaces in $\R^3$ are rigid is a classical result of Pogorelov \cite{pogorelov:book}. For $\C^1$ convex hypersurfaces this has also been established by Sen$'$kin \cite{senkin} in $\R^{m+1}$ according to V\^\i lcu \cite[Lem. 2]{vilcu}. A recent paper of Guan and Shin \cite{guan&shen} gives another proof of this fact in the $\C^2$ case. Finally see  \cite{spivak:v5} for a  classical argument for the rigidity of positively curved hypersurfaces.

\begin{proof}[Proof of Lemma \ref{lem:rigid}]
If $f$ is an isometric embedding of $M$ into $\R^{m+1}$, then so is $f\circ\phi$. Thus, by the theorems of Pogorelov and Sen$'$kin mentioned above, there exists an isometry $\rho$ of $\R^{m+1}$ such that $\rho\circ f(p)=f\circ\phi(p)=f(p^*)$ for all $p\in M$.  
In particular $\rho$ is fixed point free on $f(M)$, because $p\neq p^*$ and $f$ is injective. After a translation we may also assume that $\rho$ is linear.
Then,  
$$
\rho\(\frac{f(p)+f(p^*)}{2}\)=\frac{f(p^*)+f(p)}{2}.
$$
In other words, 
$\rho$ fixes the midpoint of each antipodal chord of $f$. Let $X$ be the affine hull of these midpoints. Then $\rho$ fixes each point of $X$. So $X$ may not intersect $f(M)$ because $\rho$ is fixed point free on $f(M)$. But, since $f(M)$ is convex, $X$ is contained in the region enclosed by $f(M)$. Consequently $X$ cannot contain any lines which means $\dim(X)=0$. So $X$ is a singleton, which implies that the midpoints of all antipodal chords of $f$ coincide. Hence $f$ is centrally symmetric.
\end{proof}

\subsection{Convex curves ($m=1$)}

Here we may identify $M$ with $\R/L\Z$. Further, similar to Section \ref{sec:ProofOf2}, we set
$$
 \ol f(t):=f(t^*)-f(t),
 $$
 where $t^*=t+L/2$, and 
 define $F\colon \R/L\Z\times[0,1]\to\R^2$ by 
   \begin{equation}\label{eq:F}
 F(t,s):=(1-s)f(t)+s\,f(t^*)=f(t)+s \,\ol f(t).
 \end{equation}
 A straight forward computation shows that 
\begin{gather*}\label{eq:Jac1}
 \|F_1\w F_2\|^2\\ \notag
 =(1-s)^2\|f'(t)\w\ol f(t)\|^2+s^2\|f'(t^*)\w\ol f(t)\|^2+2s(1-s)\big\langle f'(t)\w\ol f(t), f'(t^*)\w\ol f(t)\big\rangle,
 \end{gather*}
 where $F_1=F_1(t,s)$ and $F_2=F_2(t,s)$ denote the partial derivatives of $F$, and we may think of $\w$  as the cross product in $\R^3$.
The key observation here is that if $f$ is convex, then 
\begin{equation}\label{eq:key}
f'(t)\w\ol f(t)=-f'(t^*)\w\ol f(t).
\end{equation}
 This follows from the basic fact that when $f$ traces a convex planar curve, $f'(t)$ and $f'(t^*)$ point into the opposite sides of the line passing through $f(t)$ and $f(t^*)$, see Figure \ref{fig:rocks}, and thus $(f'(t), \ol f(t))$ and $(f'(t^*),\ol f(t))$ have opposite orientations as ordered bases of $\R^2$.

 \begin{figure}[h]
\begin{overpic}[height=1.2in]{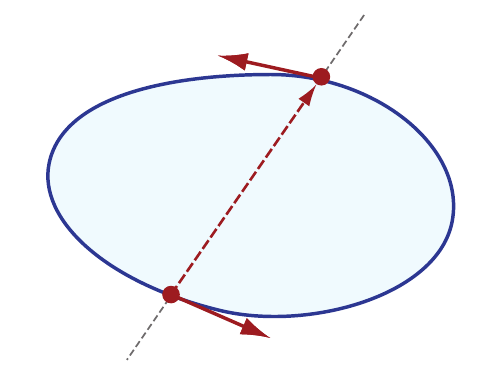}
\put(45,68){\small$f'(t^*)$}
\put(31,3){\small$f'(t)$}
\put(33,38){\small$\ol f(t)$}
\end{overpic}
\caption{}
\label{fig:rocks}
\end{figure}

The last two expressions yield that
$$
\|F_1\w F_2\|^2
 =\big((1-s)\|f'(t)\w\ol f(t)\|-s\|f'(t^*)\w\ol f(t)\|\big)^2.
$$
So  by \eqref{eq:2af}, we have
\begin{equation}\label{eq:2A}
2A(f)\le \int_0^L\int_0^1\big|(1-s)\|f'(t)\w \ol{f}(t)\| - s\|f'(t^*)\w \ol{f}(t)\|\big|\,ds\,dt.
 \end{equation}
 To estimate the above integral, we require the following fact:

\begin{lem}\label{lem:a-b-int}
Let $a,b>0$.  Then
$$
\int_0^1|(1-s)a-sb|\,ds \le\frac12 \max\{a,b\},
$$
with equality if and only if $a=b$.
\end{lem}

\begin{proof}
Computing the area under the graph of $s\mapsto |(1-s)a-sb|$, see Figure \ref{pic:graph},  

\begin{figure}[h]
\begin{overpic}[height=1.25in]{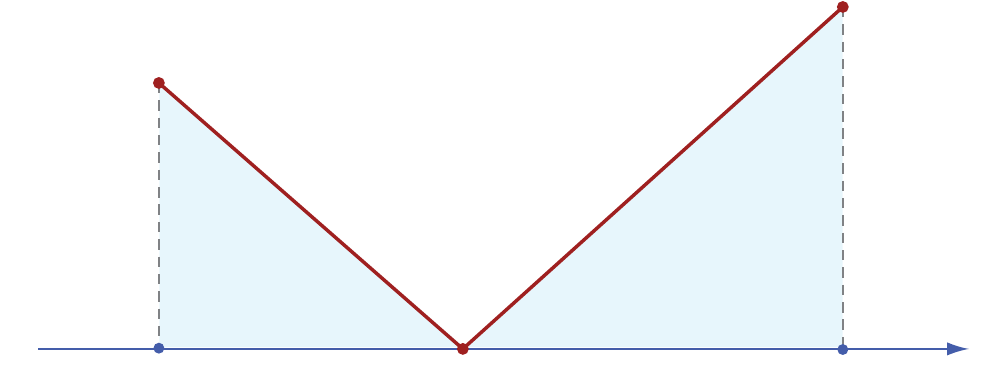}
\put(12,13){\small $a$}
\put(86,18){\small $b$}
\put(43,-2){$\frac{a}{a+b}$}
\put(15,-2){\small $0$}
\put(83.5,-2){\small $1$}
\put(98,1.5){\small $s$}
\end{overpic}
\caption{}
\label{pic:graph}
\end{figure}
\noindent
yields that
$$
\int_0^1|(1-s)a-sb|\,ds=\frac{1}{2}\(a\, \frac{a}{a+b}+b\(1-\frac{a}{a+b}\)\)=\frac{a^2+b^2}{2(a+b)}.
$$
By symmetry we may assume  $a\le b$.  Let $\lambda=b/a$.  Then
$$
\frac{a^2+b^2}{2(a+b)}=\frac a2\(\frac{1+\lambda^2}{1+\lambda}\)
=\frac a2\(\lambda -\frac{\lambda-1}{1+\lambda} \)
\le \frac a2\lambda=\frac12 b=\frac12\max\{a,b\}.
$$
and equality holds if and only if $\lambda =1$, that is when $a=b$.
\end{proof}

Using the last lemma in  \eqref{eq:2A}, and recalling that $\|f'\|\equiv 1$, we have
\begin{align*}
4A(f)&\le 2\int_0^L\int_0^1\big|(1-s)\|f'(t)\w \ol{f}(t)\| - s\|f'(t^*)\w \ol{f}(t)\|\big|\,ds\,dt\\
&\le \int_0^L\max\big\{ \|f'(t)\w \ol{f}(t)\|,  \|f'(t^*)\w \ol{f}(t)\|\big\}\,dt \\ 
&\le  \int_0^L\|\ol{f}(t)\|\,dt=TD(f) 
\end{align*}
as desired.  
Next to establish the sharpness of 
\eqref{eq:2}, note that if equality holds in \eqref{eq:2} then the first and last terms in the above expression are equal, and consequently all the intermediate terms are equal. In particular,  the equality between the integrals in the first and second lines yields that 
$$
\|f'(t)\w \ol{f}(t)\|=\|f'(t^*)\w \ol{f}(t)\|
$$ 
 via Lemma \ref{lem:a-b-int}. Consequently, the equality between the second and third lines yields that 
 $$
 \|f'(t)\w \ol{f}(t)\|=\|\ol f(t)\|=\|f'(t^*)\w \ol{f}(t)\|.
 $$
  Thus, since $\|f'\|\equiv1$, it follows that  
\begin{equation}\label{eq:key2}
\langle f'(t),\ol f(t)\rangle=0=\langle f'(t^*),\ol f(t)\rangle,
\end{equation}
which yields $f'(t)=\pm f'(t^*)$, and then  \eqref{eq:key} ensures that $f'(t)=-f'(t^*)$.  Consequently $o(t):=(f(t)+f(t^*))/2$ does not depend on $t$, i.e., $o'(t)=0$. Now if  we set $o:=o(t)$, then we have
$$
\frac{d}{dt}\|f(t)-o\|^2=\frac12\frac{d}{dt}\|\ol f(t)\|^2
=\la f'(t)-f'(t^*), \ol f(t)\ra=0,
$$
where the last equality again follows from \eqref{eq:key2}.
So  $f$ traces a circle centered at $o$.

  \section{Sharpness of  \eqref{eq:1} for Simple Curves}\label{sec:example}
Here we construct a one-parameter family of simple closed  curves $f_n\colon\R/L_n\Z\to\R^2$ parametrized by arc length 
 such that
\begin{equation}\label{eq:afn}
\lim_{n\to\infty}\frac{1}{A(f_n)}\int_0^{L_n}\|\ol f_n(t)\|\,dt=2,
\end{equation}
where $\ol f_n(t):=f_n(t^*)-f_n(t)$ and $t^*:=t+L_n/2$. This 
shows that  the constant $2$ in \eqref{eq:1} is in general sharp even for simple curves. Note that by \eqref{eq:1}, which we have already established, the above limit is  always $\geq 2$. Thus it suffices to find  $f_n$ such that 
this limit is $\leq 2$.
To this end, for $n=1,2,\dots$, let  each $f_n$  trace with unit speed a horseshoe shaped curve which consists of a pair of rectangular parts joined by concentric semicircles as depicted in Figure \ref{fig:horseshoe}.
 \begin{figure}[h]
\begin{center}
\begin{overpic}[height=1.3in]{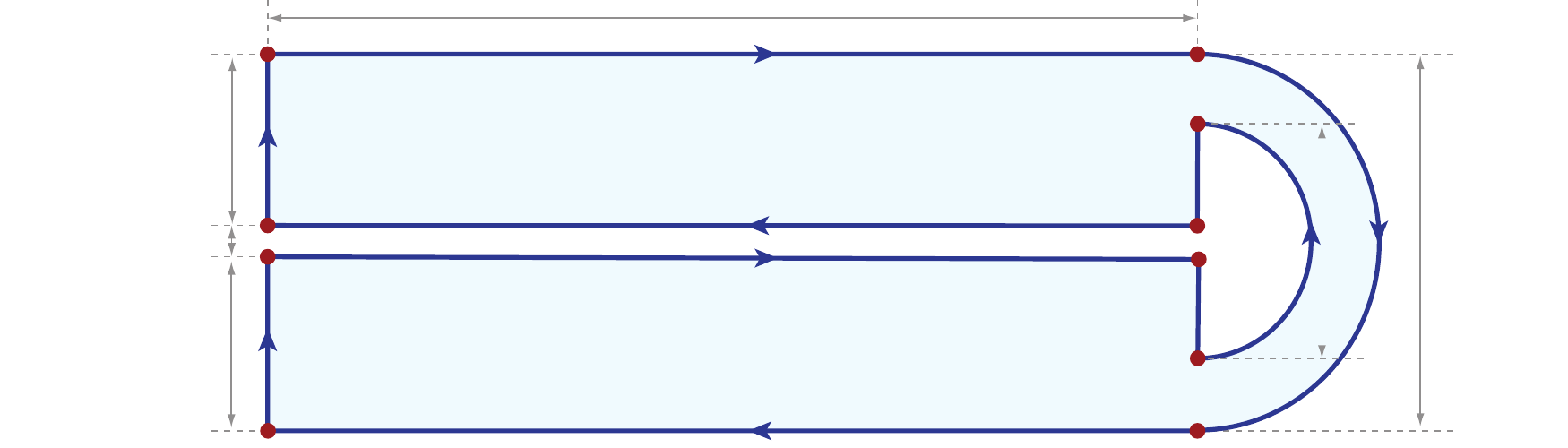}
    \put(10,6){$\frac{1}{n}$}
    \put(10,12.5){$\frac{1}{n^2}$}
    \put(10,19){$\frac{1}{n}$}
    \put(47,28){\small$1$}
      \put(84.5,12.5){\small $2r$}
     \put(91,12.5){\small $2R$}
     \put(18,22.5){$a$}
     \put(74,22.5){$b$}
     \put(18,15){$c$}
     \put(74,15){$d$}
     \put(18,9.5){$a^*$}
     \put(73.5,9.5){$b^*$}
     \put(18,2){$c^*$}
     \put(74,2){$d^*$}
    \end{overpic}
\caption{}\label{fig:horseshoe}
\end{center}
\end{figure}
The rectangular parts here have constant length $1$ and height $1/n$. Further, the vertical separation distance between them is $1/n^2$. In particular note that
$$
A(f_n)\geq \frac2n.
$$
Next note that if $R$ denotes the radius of the big semicircle, then we have
$$
R=\frac{1}{n}+\frac{1}{2n^2}\leq \frac{2}{n}.
$$
Let $a$, $b$, $c$, $d$ denote the corners of the top rectangle, and $a^*$, $b^*$, $c^*$, $d^*$ be the  corners of the bottom rectangle, as indicated in the figure. We want the corresponding points in these two sets  to be antipodal, i.e., $a^*=a+L_n/2$ and so on.  To this end it suffices to choose the radius $r$ of the small semicircle so that the length of the arc $bd^*$ is equal to that of the arc $b^*d$ (with respect to the orientation of the curve as indicated in the figure). The former quantity is $\pi R$ while the latter is $\pi r +2r-1/n^2$. Setting these values equal to each other, we obtain
$$
r:=\frac{\pi R+1/n^2}{\pi+2}.
$$
Now  it follows that for every point $x\in ab$, $x^*$ lies directly below it on  $a^*b^*$, and for every $x\in dc$, $x^*$ lies directly below it on  $d^*c^*$. Thus, if we let $C_n$ denote the trace of $f_n$, and $\ol C_n:=ab\cup dc\cup a^*b^*\cup d^*c^*$ be the  horizontal portions of $C_n$, then 
$$
\| \ol f_n\|=\frac{1}{n}+\frac{1}{2n^2}
$$
on $\ol C_n$. So we obtain the following estimate
$$
\int_{\ol C_n}\| \ol f_n\|= \Length(\ol C_n)\left(\frac{1}{n}+\frac{1}{2n^2}\right)=4 \left(\frac{1}{n}+\frac{1}{2n^2}\right)\leq 2 A(f_n) +\frac{2}{n^2}.
$$
Further we have
$$
\int_{C_n-\ol C_n}\|\ol f_n\|\leq \Length(C_n-\ol C_n)\max_{C_n-\ol C_n} \|\ol f_n\|\leq\left(4R+2\pi R\right)\sqrt{5}R\leq 23 R^2\leq \frac{92}{n^2}.
$$
The last two inequalities show that
$$
\int_0^{L_n}\| \ol f_n(t)\|\,dt=\int_{C_n} \|\ol f_n\|=\int_{\ol C_n}\| \ol f_n\|+\int_{C_n-\ol C_n}\|f_n\|\leq 2 A(f_n) +\frac{94}{n^2}.
$$ 
 So  it follows that
$$
\frac{1}{A(f_n)}\int_0^{L_n}\|\ol f_n(t)\|\,dt\leq 2+\frac{47}{n},
$$
which shows that the left hand side of  \eqref{eq:afn} is $\leq 2$ as desired.

\section*{Acknowledgements} 
We are grateful to Jaigyoung Choe who first suggested to us that inequality \eqref{eq:2} should hold for planar curves, and thus provided the initial stimulus for this work. We also thank Igor Belegradek for locating the reference \cite{senkin}. Finally, thanks to the anonymous referee for suggesting improvements to an earlier draft of this work.

\bibliographystyle{abbrv}
\bibliography{references}

\begin{thebibliography}{10}

\bibitem{abrams:circle}
A.~Abrams, J.~Cantarella, J.~H.~G. Fu, M.~Ghomi, and R.~Howard.
\newblock Circles minimize most knot energies.
\newblock {\em Topology}, 42(2):381--394, 2003.

\bibitem{BP}
T.~F. Banchoff and W.~F. Pohl.
\newblock A generalization of the isoperimetric inequality.
\newblock {\em J. Differential Geometry}, 6:175--192, 1971/72.

\bibitem{Crofton:Prob}
M.~W. Crofton.
\newblock {Probability}.
\newblock In {\em {Encyclopaedia Britannica}}, volume~19, pages 768--788. {A \&
  C Black}, 9th edition, 1885.

\bibitem{Crofton:local}
M.~W. Crofton.
\newblock On the theory of local probability.
\newblock {\em Phil. Trans. Roy. Soc. London}, 159:181--199, 1968.

\bibitem{EHL}
P.~Exner, E.~M. Harrell, and M.~Loss.
\newblock Inequalities for means of chords, with application to isoperimetric
  problems.
\newblock {\em Lett. Math. Phys.}, 75(3):225--233, 2006.

\bibitem{FHW:mobius}
M.~H. Freedman, Z.-X. He, and Z.~Wang.
\newblock M\"obius energy of knots and unknots.
\newblock {\em Ann. of Math. (2)}, 139(1):1--50, 1994.

\bibitem{gabor:mean}
L.~G{\'a}bor.
\newblock On the mean length of the chords of a closed curve.
\newblock {\em Israel J. Math.}, 4:23--32, 1966.

\bibitem{guan&shen}
P.~Guan and X.~Shen.
\newblock A rigidity theorem for hypersurfaces in higher dimensional space
  forms.
\newblock {\em arXiv:1306.1581v1}, 2013.

\bibitem{hayashi}
T.~Hayashi.
\newblock The extremal chords of an oval.
\newblock {\em T\^{o}hoku Math. J.}, 22:387----393, 1923.

\bibitem{krantz&parks:book2}
S.~G. Krantz and H.~R. Parks.
\newblock {\em Geometric integration theory}.
\newblock Cornerstones. Birkh\"auser Boston Inc., Boston, MA, 2008.

\bibitem{morgan:book}
F.~Morgan.
\newblock {\em Geometric measure theory}.
\newblock Elsevier/Academic Press, Amsterdam, fourth edition, 2009.
\newblock A beginner's guide.

\bibitem{ohara:energy1}
J.~O'Hara.
\newblock Family of energy functionals of knots.
\newblock {\em Topology Appl.}, 48(2):147--161, 1992.

\bibitem{or:degree}
E.~Outerelo and J.~M. Ruiz.
\newblock {\em Mapping degree theory}, volume 108 of {\em Graduate Studies in
  Mathematics}.
\newblock American Mathematical Society, Providence, RI, 2009.

\bibitem{pogorelov:book}
A.~V. Pogorelov.
\newblock {\em Extrinsic geometry of convex surfaces}.
\newblock American Mathematical Society, Providence, R.I., 1973.
\newblock Translated from the Russian by Israel Program for Scientific
  Translations, Translations of Mathematical Monographs, Vol. 35.

\bibitem{Pohl:Integral-formulas}
W.~F. Pohl.
\newblock Some integral formulas for space curves and their generalization.
\newblock {\em Amer. J. Math.}, 90:1321--1345, 1968.

\bibitem{Rado:Lemma}
T.~Rad{\'o}.
\newblock A lemma on the topological index.
\newblock {\em Fund. Math.}, 27:212--225, 1936.

\bibitem{Rado:Surface-Area}
T.~Rad{\'o}.
\newblock The isoperimetric inequality and the {L}ebesgue definition of surface
  area.
\newblock {\em Trans. Amer. Math. Soc.}, 61:530--555, 1947.

\bibitem{Rado:Book}
T.~Rad{\'o}.
\newblock {\em Length and {A}rea}.
\newblock American Mathematical Society Colloquium Publications, vol. 30.
  American Mathematical Society, New York, 1948.

\bibitem{santalo}
L.~A. Santal{\'o}.
\newblock {\em Integral geometry and geometric probability}.
\newblock Addison-Wesley Publishing Co., Reading, Mass.-London-Amsterdam, 1976.
\newblock With a foreword by Mark Kac, Encyclopedia of Mathematics and its
  Applications, Vol. 1.

\bibitem{scott}
P.~R. Scott and P.~W. Awyong.
\newblock Inequalities for convex sets.
\newblock {\em JIPAM. J. Inequal. Pure Appl. Math.}, 1(1):Article 6, 6 pp.
  (electronic), 2000.

\bibitem{senkin}
E.~P. Sen$'$kin.
\newblock Rigidity of convex hypersurfaces.
\newblock {\em Ukrain. Geometr. Sb.}, (12):131--152, 170, 1972.

\bibitem{spivak:v5}
M.~Spivak.
\newblock {\em A comprehensive introduction to differential geometry. {V}ol.
  {V}}.
\newblock Publish or Perish Inc., Wilmington, Del., second edition, 1979.

\bibitem{vilcu}
C.~V{\^{\i}}lcu.
\newblock On typical degenerate convex surfaces.
\newblock {\em Math. Ann.}, 340(3):543--567, 2008.

\end{thebibliography}

\end{document}